\documentclass[12pt]{amsart}

\usepackage[all,matrix,arrow]{xy}
\usepackage{mathrsfs, amssymb, array, amsmath, upgreek, amsfonts}
\usepackage{amsthm,textcomp}
\usepackage[mathcal]{euscript}
\usepackage{float}

\usepackage{hyperref}
\usepackage{tabularx,multirow,makecell,longtable}

\makeatletter
\@addtoreset{equation}{section}
\makeatother

\newcommand{\PP}{\mathbb P}

\newcommand{\QQ}{\mathbb Q}
\newcommand{\ZZ}{\mathbb Z}

\newcommand{\OOO}{\mathscr{O}}

\newcommand{\KKK}{{\mathscr{K}}}

\newcommand{\chit}{\chi_{\mathrm{top}}}
\newcommand{\Bim}{\operatorname{Bim}}

\newcommand{\Supp}{\operatorname{Supp}}

\newcommand{\Aut}{\operatorname{Aut}}

\newcommand{\cc}{\operatorname{c}}

\newcommand{\GL}{\operatorname{GL}}

\def \ge {\geqslant}
\def \le {\leqslant}

\theoremstyle{plain}
\newtheorem{theorem}[subsection]{Theorem}
\newtheorem{lemma}[subsection]{Lemma}
\newtheorem{proposition}[subsection]{Proposition}
\newtheorem{corollary}[subsection]{Corollary}

\theoremstyle{definition}

\newtheorem{remark}[subsection]{Remark}

\title{Bounded automorphism groups of compact complex surfaces}

\author{Yu.~G.~Prokhorov, \quad C.~A.~Shramov}

\address{\emph{Yuri Prokhorov}
\newline
\textnormal{Steklov Mathematical Institute of RAS,
8 Gubkina street, Moscow 119991, Russia.
}
\newline
\textnormal{\texttt{prokhoro@mi-ras.ru}}}

\address{\emph{Constantin Shramov}
\newline
\textnormal{Steklov Mathematical Institute of RAS,
8 Gubkina street, Moscow 119991, Russia.
}
\newline
\textnormal{\texttt{costya.shramov@gmail.com}}}

\thanks{This work was performed at the Steklov International Mathematical Center
and supported by the Ministry of Science and Higher Education of the
Russian Federation (agreement no. 075-15-2019-1614).}

\date{}

\begin{document}

\maketitle

\begin{abstract}
We classify compact complex surfaces whose groups of bimeromorphic selfmaps
have bounded finite subgroups. We also
prove that the stabilizer of a point in the automorphism group of a compact complex surface of zero Kodaira dimension,
as well as the stabilizer of a point in the automorphism group of an arbitrary compact K\"ahler manifold of non-negative
Kodaira dimension, always has bounded finite subgroups.
\end{abstract}

\tableofcontents

\section{Introduction}

One says that a group $\Gamma$ has  \emph{bounded finite subgroups}
if there exists a constant $B=B(\Gamma)$ such that every finite subgroup of $\Gamma$
has order at most $B$.
Otherwise one says that $\Gamma$ has  \emph{unbounded finite subgroups}.
In many interesting cases automorphism groups or groups of
birational selfmaps of algebraic varieties have bounded finite subgroups.
For instance, this is the case for non-uniruled varieties with vanishing irregularity
over fields of characteristic zero (see~\mbox{\cite[Theorem~1.8(i)]{ProkhorovShramov-Bir}});
for varieties over number fields (see~\mbox{\cite[Theorem~1.4]{ProkhorovShramov-Bir}} and~\mbox{\cite[Theorem~1.1]{Birkar}});
for non-trivial Severi--Brauer surfaces over fields of characteristic zero that contain all roots of unity (see~\mbox{\cite[Corollary~1.5]{ShramovVologodsky}}).

The main purpose of this paper is to prove the following assertion
that classifies compact complex surfaces whose groups of bimeromorphic
selfmaps have unbounded finite subgroups.

\begin{theorem}[{cf. \cite[Lemma~3.5]{Prokhorov-Shramov-3folds}}]
\label{theorem:BFS-for-surfaces}
Let $S$ be a compact complex surface of non-negative Kodaira dimension.
Suppose that the group~\mbox{$\Bim(S)$} of bimeromorphic selfmaps of $S$
has unbounded finite subgroups. Then~$S$ is bimeromorphic to a surface of one of the following types:
\begin{itemize}
\item a complex torus;

\item a bielliptic surface;

\item a Kodaira surface;

\item a surface of Kodaira dimension $1$.
\end{itemize}
Moreover, in the first three cases the group $\Bim(S)$ always has
unbounded finite subgroups. In the fourth case the group~\mbox{$\Bim(S)$}
has bounded finite subgroups if and only if this holds for its subgroup
that consists of all selfmaps preserving the fibers of the pluricanonical
fibration of~$S$.
\end{theorem}

One of the main steps in the proof of Theorem~\ref{theorem:BFS-for-surfaces}
is the following assertion which, as we think, is interesting on its own.

\begin{proposition}\label{proposition:elliptic-fibration-base-subgroup}
Let $S$ be a minimal compact complex surface of Kodaira dimension~$1$.
Consider the pluricanonical fibration $\phi\colon S\to C$.
Then the image of the group $\Aut(S)$ in $\Aut(C)$ is finite.
\end{proposition}

Also, in this paper we prove the following result.
For a complex manifold $X$ and a point $P\in X$,
by $\Aut(X;P)$ we denote the stabilizer of $P$ in the group~\mbox{$\Aut(X)$}.

\begin{proposition}
\label{proposition:stabilizer-kappa-0}
There exists a constant $B$ such that for every compact complex surface
$S$ of Kodaira dimension $0$, every point $P\in S$, and every finite subgroup~\mbox{$G\subset \Aut(S;P)$}
the order of the group $G$ is at most~$B$.
\end{proposition}

\begin{remark}
Let $S$ be a compact complex surface of Kodaira dimension~$1$, and let $P$ be a point on $S$.
We do not know whether the group~\mbox{$\Aut(S;P)$} always has bounded finite subgroups.
It is easy to show that this holds at least when the fiber of the pluricanonical fibration of $S$
passing through $P$ is non-singular and reduced at $P$.
\end{remark}

Under an additional assumption
that a manifold is K\"ahler we can prove boundedness of finite subgroups
for stabilizers of points in the automorphism groups in arbitrary dimension.

\begin{theorem}\label{theorem:Kahler-stabilizer}
Let $X$ be a compact K\"ahler manifold of non-negative Kodaira dimension, and let $P$
be a point on $X$. Then the group~\mbox{$\Aut(X;P)$} has bounded finite subgroups.
\end{theorem}

We will prove one more assertion whose structure is similar to
that of Proposition~\ref{proposition:stabilizer-kappa-0};
it completes some of the results of the paper~\cite{ProkhorovShramov-CCS}.

\begin{proposition}[{cf. \cite[Corollary~8.10]{ProkhorovShramov-CCS}}]
\label{proposition:uniformly-Jordan}
There exists a constant~$J$ such that for every compact complex surface $S$
of Kodaira dimension~$0$ and every finite subgroup~\mbox{$G\subset \Bim(S)$}
there is an abelian subgroup of index at most~$J$ in~$G$.
\end{proposition}

In Section~\ref{section:preliminaries} we prove several auxiliary assertions.
In Section~\ref{section:elliptic-surfaces} we study automorphism groups of elliptic fibrations and prove Proposition~\ref{proposition:elliptic-fibration-base-subgroup}.
In Section~\ref{section:proof} we prove Theorem~\ref{theorem:BFS-for-surfaces}.
In Section~\ref{section:stabilizer} we prove Proposition~\ref{proposition:stabilizer-kappa-0}.
In Section~\ref{section:Kahler} we prove Theorem~\ref{theorem:Kahler-stabilizer}.
In Section~\ref{section:Jordan} we prove Proposition~\ref{proposition:uniformly-Jordan}.

\smallskip
We use the following notation and conventions.
A \emph{complex manifold} is a smooth irreducible complex space. A
\textit{morphism} is a holomorphic map of complex spaces; a
\textit{fibration} is a morphism with connected fibers of positive dimension.
A \emph{typical fiber} of a fibration $\phi\colon X\to Y$ is a fiber over a point
of some non-empty subset of the form~\mbox{$Y\setminus\Sigma$}, where $\Sigma$ is a closed
(analytic) subset in~$Y$.
By $\KKK_X$ we denote the canonical line bundle on a (compact) complex manifold $X$, and
by $\varkappa(X)$ we denote the Kodaira dimension of~$X$.
By $T_P(X)$ we denote the tangent space to a complex manifold $X$ at a point $P\in X$.
Given a complex manifold~$X$ and a morphism $\phi\colon X\to Y$ which is equivariant
with respect to the group $\Aut(X)$, we denote by $\Aut(X)_\phi$ the group that consists
of all automorphisms mapping every fiber of $\phi$ to itself.

\smallskip
We are grateful to F.\,Campana, S.\,Nemirovski, and W.\,Sawin
for useful discussions. Special thanks go to the referee for his remarks.

\section{Preliminaries}
\label{section:preliminaries}

In this section we prove some auxiliary results on compact complex surfaces.
The following assertion is well-known.

\begin{lemma}\label{lemma:two-intersecting-curves}
Let $S$ be a compact complex surface. Suppose that there are two
divisors~$C_1$ and~$C_2$ on $S$ such that
$C_1^2\ge 0$ and $C_1\cdot C_2>0$.
Then the surface $S$ is projective.
\end{lemma}
\begin{proof}
For $n\gg 0$ one has
$$
(nC_1+C_2)^2=n^2C_1^2+2nC_1\cdot C_2+C_2^2>0,
$$
so that $S$ is projective by~\mbox{\cite[Theorem~IV.6.2]{BHPV-2004}}.
\end{proof}

Recall that a compact complex surface $S$ is called  \emph{minimal}
if it does not contain smooth rational curves with self-intersection~$-1$.
For every compact complex surface $S$ there exists a minimal surface~$S'$
bimeromorphic to $S$, see for instance~\mbox{\cite[Theorem~III.4.5]{BHPV-2004}}.
Furthermore, if~\mbox{$\varkappa(S)\ge 0$}, then the minimal surface $S'$
is unique, see~\mbox{\cite[Proposition~III.4.6]{BHPV-2004}}; in this case there is a bimeromorphic
morphism~\mbox{$\pi\colon S\to S'$}.
Recall also that there exists a Kodaira--Enriques classification
of minimal compact complex surfaces, see~\mbox{\cite[Chapter~VI]{BHPV-2004}}.

\begin{lemma}[{see for instance \cite[Proposition~3.5]{ProkhorovShramov-CCS}}]
\label{lemma:Bir-vs-Aut}
Let~$S$ be a non-ruled minimal compact complex surface. Then~\mbox{$\Bim(S)=\Aut(S)$}.
\end{lemma}

\begin{lemma}\label{lemma:faithful-on-fiber}
Let $S$ be a complex manifold, and let $\phi\colon S\to C$ be a morphism that is equivariant
with respect to the group $\Aut(S)$. Let
$$
F=\phi^*(c),\quad c\in C,
$$
be a fiber of the morphism $\phi$.
Suppose that there is an irreducible component $F_1\subset \Supp(F)$ of multiplicity~$1$ in $F$.
Then every finite subgroup of $\Aut(S)_{\phi}$ acts faithfully on $F$.
\end{lemma}
\begin{proof}
Suppose that the induced action on~$F$ of some non-trivial finite subgroup $G\subset\Aut(S)_\phi$
is trivial; in particular, $G$ preserves the irreducible component $F_1$
and acts trivially on it. Choose a point~\mbox{$P\in F_{1}$} such that the morphism $\phi$ is smooth at~$P$.
The group $G$ acts non-trivially in the tangent space $T_P(S)$,
see for instance~\mbox{\cite[\S2.2]{Akhiezer}} or~\mbox{\cite[Corollary~4.2]{ProkhorovShramov-CCS}}.
On the other hand, $G$ acts trivially on the subspace~\mbox{$T_P(F_{1})\subset T_P(S)$}.
Since the morphism $\phi$ is smooth at~$P$, the differential
$$
d\phi\colon T_P(S)\longrightarrow T_{\phi(P)}(C)
$$
is a surjective linear map whose kernel is identified with $T_P(F_{1})$.
Moreover, this map is $G$-equivariant, where the action of $G$ on $C$ is taken to be trivial.
Since $G$ acts trivially on the tangent space
$T_{\phi(P)}(C)$, we conclude that the action of $G$ on the tangent space $T_P(S)$ is trivial as well.
The obtained contradiction shows that the action of $G$ on $F$ cannot be trivial.
\end{proof}

Recall that for a primary Kodaira surface the algebraic reduction
gives a morphism onto an elliptic curve, and all of its fibers are elliptic curves.
This morphism is equivariant with respect to the automorphism group of the surface.
For a secondary Kodaira surface the algebraic reduction is a morphism
to a rational curve whose typical fiber is an elliptic curve.

\begin{lemma}\label{lemma:primary-Kodaira-Aut}
Let $S$ be a primary (respectively, secondary) Kodaira surface, and let~\mbox{$\phi\colon S\to C$}
be its algebraic reduction. Then the image~$\Gamma$ of the group~\mbox{$\Aut(S)$} in $\Aut(C)$ has order at most $6$ (respectively,~$24$).
\end{lemma}
\begin{proof}
If $S$ is a primary Kodaira surface, then the group~$\Gamma$ does not contain elements that
act on the elliptic curve~$C$ by translations, see~\mbox{\cite[Corollary~3.3]{Shramov-Elliptic}}.
Since the automorphism group of an elliptic curve is a semi-direct product of the subgroup of translations
with a cyclic group of order at most~$6$, we conclude that  the order of~$\Gamma$ is at most~$6$.

Let $S$ be a secondary Kodaira surface, and let~\mbox{$\phi\colon S\to \mathbb{P}^1$}
be its algebraic reduction. Then $\phi$ has either  $3$ or $4$ multiple fibers, see for instance
the proof of \cite[Lemma~7.1]{ProkhorovShramov-CCS}.
Thus $\Gamma$ is isomorphic to a subgroup of the symmetric group on four elements.
In particular, the order of $\Gamma$ is at most~$24$.
\end{proof}

\begin{lemma}\label{lemma:primary-Kodaira-fiber}
Let $S$ be a primary Kodaira surface, and let~\mbox{$\phi\colon S\to C$}
be its algebraic reduction. Let~\mbox{$G\subset \Aut(S)_\phi$} be a non-trivial subgroup
of finite order. Then $G$ acts on $S$ without fixed points.
\end{lemma}
\begin{proof}
Without loss of generality, we may assume that the group $G$ is cyclic.
Let $\gamma$ be its generator. Suppose that $G$ has a fixed point~\mbox{$P\in S$}.
Let $F$ be the fiber of $\phi$ passing through $P$.
Recall that the fiber~$F$ is irreducible and non-multiple, see~\mbox{\cite[\S\,V.5]{BHPV-2004}}.
By Lemma~\ref{lemma:faithful-on-fiber} the automorphism~$\gamma$ acts non-trivially
on the elliptic curve $F$. Suppose that~$\gamma$ fixes some point $P$ on $F$.
Note that the morphism $\phi$ is smooth at~$P$
(as well as at any other point on $S$).
Arguing as in the proof of Lemma~\ref{lemma:faithful-on-fiber}, we see that the differential
$$
d\phi\colon T_P(S)\longrightarrow T_{\phi(P)}(C)
$$
allows us to identify the two-dimensional representation $T_P(S)$ of the cyclic group $G$
with the direct sum~\mbox{$T_P(F)\oplus T$},
where the action of~$G$ on the subspace $T\cong T_{\phi(P)}(C)$ is trivial.
By~\mbox{\cite[Corollary~4.7]{ProkhorovShramov-CCS}} there exists a (compact) curve $D\subset S$ that consists of
fixed points of~$\gamma$, such that $D$ passes through $P$ and the tangent space to $D$ at $P$
coincides with the subspace $T\subset T_P(S)$.
Therefore, we have $\phi(D)=C$. Since the surface $S$ is not projective, this gives a contradiction
with Lemma~\ref{lemma:two-intersecting-curves}.
\end{proof}

We will use the following classical theorem proved by H.\,Minkowski.

\begin{theorem}[{see for instance~\cite[Theorem~1]{Serre2007}}]
\label{theorem:Minkowski}
For every positive integer $n$ the group~{$\GL_n(\mathbb{Q})$} has bounded finite subgroups.
\end{theorem}

\begin{corollary}\label{corollary:torus-stabilizer}
For every positive integer $n$ there exists a constant~\mbox{$B_T(n)$} such that for every $n$-dimensional complex torus~$S$,
every point~$P$ on~$S$, and every finite subgroup~\mbox{$G\subset \Aut(S;P)$}, the order of~$G$
is at most~\mbox{$B_T(n)$}.
\end{corollary}
\begin{proof}
It is well-known that
$\Aut(S;P)$ is isomorphic to a subgroup of~\mbox{$\GL_{2n}(\mathbb{Z})$}; see for instance  \cite[Theorem~8.4]{ProkhorovShramov-CCS}. Therefore, the assertion follows from Theorem~\ref{theorem:Minkowski}.
\end{proof}

\begin{corollary}
\label{corollary:K3}
There exists a constant $B_{K3}$ such that for every compact complex surface $S$
which is either a $K3$ surface or an Enriques surface, and for every finite subgroup
$G\subset \Aut(S)$, the order of~$G$
is at most~$B_{K3}$.
\end{corollary}

\begin{proof}
Let $S$ be either a $K3$ surface or an Enriques surface.
Consider the representation of $\Aut(S)$ in the cohomology group
\[
\rho\colon \Aut(S)\to \GL\big(H^*(S,\QQ)\big).
\]
In both cases we have to consider the dimension of the vector space~\mbox{$H^*(S,\QQ)$} does not exceed~$24$,
see for instance \cite[Table~10]{BHPV-2004}.
By Theorem~\ref{theorem:Minkowski} the
orders of finite subgroups in the image of the representation $\rho$ are bounded by some constant that does not depend on~$S$.
On the other hand, if $S$ is a $K3$ surface, then the representation $\rho$ is faithful
(see for instance \cite[Proposition~15.2.1]{Huybrechts}).
If $S$ is an Enriques surface, then the kernel of~$\rho$ contains at most~$4$ elements
(see~\cite{MukaiNamikawa}).
\end{proof}

Finite groups acting on $K3$ surfaces are well studied,
see \cite{Nikulin} and references therein. This classification can be 
used to obtain an adequate bound for the constant 
$B_{K3}$ from Corollary~\ref{corollary:K3}.

The following theorem is well-known.

\begin{theorem}[{see for instance \cite[Corollary~14.3]{Ueno1975}}]
\label{theorem:general-type}
Let $S$ be a compact complex surface of Kodaira dimension $2$. Then the group $\Aut(S)$ is finite.
\end{theorem}

\section{Elliptic surfaces}
\label{section:elliptic-surfaces}

In this section we prove Proposition~\ref{proposition:elliptic-fibration-base-subgroup}.

To describe fibers of elliptic fibrations we will use the standard notation for
types of degenerate fibers in the Kodaira classification, see~\mbox{\cite[\S\,V.7]{BHPV-2004}} or \cite[\S\,I.4]{Miranda1989}.
Recall that a fibration $\phi\colon S\to C$ with compact fibers, where $S$ is a smooth
(but possibly non-compact) complex surface and $C$ is a smooth (but possibly non-compact)
curve, is called \emph{relatively minimal}
if the fibers of $\phi$ do not contain smooth rational curves with self-intersetion~$-1$.

Given an elliptic fibration $\phi\colon S\to C$, one can consider the function~$J$
that associates to a point $c\in C$ the value of the $j$-invariant of the fiber~\mbox{$F=\phi^*c$}, provided that
the fiber $F$ is smooth. It is straightforward to see that $J$ is a meromorphic
function on~$C$.

\begin{lemma}
\label{lemma:j-poles}
Let $\Delta\subset\mathbb{C}$ be the unit disc.
Let $S$ be a smooth (non-compact) complex surface, and let~\mbox{$\phi\colon S\to \Delta$}
be a fibration with compact fibers whose typical fiber is an elliptic curve.
Suppose that~$\phi$ is relatively minimal. Consider the fiber  $F=\phi^*\mathbf{0}$ over the point~\mbox{$\mathbf{0}\in\Delta$}.
Suppose that the fiber $F$ is non-multiple, and has type $I_b$ or~$I_b^*$ for some~\mbox{$b\ge 1$}.
Then the function $J$ has a pole at the point~$\mathbf{0}$.
\end{lemma}

\begin{proof}
Since the fiber $F$ is non-multiple, at least one of its irreducible components is reduced  (see~\mbox{\cite[\S\,V.7]{BHPV-2004}}).
Hence in a small neighborhood of the point $\mathbf{0}$ the fibration $\phi$ admits an analytic section.
Thus we can assume that $\phi$ is taken to the Weierstrass form, see \cite[Lecture~II]{Miranda1989}.
Now the assertion follows from~\mbox{\cite[Table~IV.3.1]{Miranda1989}}
(cf.~\mbox{\cite[\S\,V.10, Table~6]{BHPV-2004}} and \cite[Proposition~VI.1.1]{Miranda1989}).
\end{proof}

Recall that a fibration is called \emph{isotrivial} if its fibers over the points of some
dense open subset are isomorphic to each other.

\begin{lemma}[{cf.~\cite[\S2.1]{Sawon}}]
\label{lemma:isotrivial-fibers}
Let $\Delta\subset\mathbb{C}$ be the unit disc.
Let $S$ be a smooth (non-compact) complex surface, and let~\mbox{$\phi\colon S\to \Delta$}
be a fibration with compact fibers whose typical fiber is an elliptic curve.
Suppose that  $\phi$ is relatively minimal and isotrivial. Consider the fiber  $F=\phi^*\mathbf{0}$ over the point $\mathbf{0}\in\Delta$.
If $F$ is a non-multiple fiber, then it cannot be of type $I_b$ or~$I_b^*$ for $b\ge 1$;
if $F$ is a multiple fiber, then it can only have type ${}_mI_0$ for some $m\ge 2$.
\end{lemma}
\begin{proof}
Since the fibration $\phi$ is isotrivial, the $j$-invariant of its fibers is a constant
function on the set $\Delta\setminus\Lambda$, where  $\Lambda$ is the set of
images of all multiple fibers of $\phi$. In particular, this function cannot have a pole
at a point of $\Delta\setminus\Lambda$. Therefore, the assertion about non-multiple fibers
follows from Lemma~\ref{lemma:j-poles}.

Now suppose that the fiber $F$ is multiple.
According to~\cite[\S\,V.7]{BHPV-2004}, it can only have type~${}_mI_r$, where $r\ge 0$ and~\mbox{$m\ge 2$}.
We need to exclude the case $r\ge 1$.
To do this, apply the standard construction from~\mbox{\cite[\S\,III.10]{BHPV-2004}} or~\mbox{\cite[pp.~571--572]{Kodaira-1963}}:
it gives an isotrivial relatively minimal elliptic fibration over~$\Delta$
with a non-multiple fiber of type~$I_{r'}$ for some~\mbox{$r'>0$}. The latter gives a contradiction with what we already proved above.
\end{proof}

We will need a detailed information on automorphism groups of compact complex surafces
of Kodaira dimension~$1$
(similar computations were used in the proofs of~\cite[Lemma~8.2]{ProkhorovShramov-CCS} and~\cite[Lemma~3.3]{Prokhorov-Shramov-3folds}).

\begin{lemma}
\label{lemma:elliptic-fibration-base-subgroup-P1}
Let $S$ be a minimal compact complex surface of Kodaira dimension~$1$.
Consider the pluricanonical fibration $\phi\colon S\to C$ and suppose that~\mbox{$C\cong\PP^1$}.
Then the image $\Gamma$ of the group~\mbox{$\Aut(S)$} in $\Aut(C)$ is finite.
\end{lemma}

\begin{proof}
The group $\Gamma$ acts faithfully on $\PP^1$.
Assume that $\Gamma$ is infinite. Then there exists a point~\mbox{$P\in\PP^1$}
with an infinite $\Gamma$-orbit $\Xi_P\subset\PP^1$ (indeed, if the orbit of every point
is finite, $\Gamma$ contains a subgroup of finite index that fixes three points on~$\PP^1$).
The fibers of $\phi$ over all points of $\Xi_P$ are isomorphic to each other.
Since the $j$-invariant of the fibers of $\phi$
is a meromorphic function on $\PP^1$,
we conclude that this function is constant, so that the elliptic fibration
$\phi$ is isotrivial. In particular, by Lemma~\ref{lemma:isotrivial-fibers}
all multiple fibers of $\phi$
have type ${}_mI_0$.

Suppose that $\phi$ has at least one singular non-multiple fiber. Since the group $\Gamma$ is infinite,
$\phi$ has at most two singular non-multiple fibers; indeed, otherwise $\Gamma$ would have a finite invariant subset
in~$\PP^1$ of cardinality at least~$3$, which is impossible for a subgroup of~\mbox{$\Aut(\PP^1)$}.
By Lemma~\ref{lemma:isotrivial-fibers}
and Kodaira classification the topological Euler characteristic of any
singular non-multiple fiber of $\phi$ is positive and does not exceed $10$.
Thus, one has
$$
0<\chit(S)\le 20.
$$
By Noether's formula we obtain
\[
\chi(\OOO_S)=\frac{1}{12}\left(\cc_1(S)^2+\chit(S)\right)=\frac{\chit(S)}{12},
\]
which implies that $\chi(\OOO_S)=1$.

Let $F_i$ be all multiple fibers of $\phi$ (considered with reduced structure), and let $m_i$
be the multiplicity of $F_i$. By the canonical bundle formula (see~\cite[Theorem~V.12.1]{BHPV-2004})
one has
\[
\KKK_S\sim\phi^*\left(\KKK_{\PP^1}\otimes\mathcal{L}\right)\otimes\OOO_S\left(\sum (m_i-1) F_i\right),
\]
where $\mathcal{L}$ is some line bundle of degree
$\chi(\OOO_S)=1$ on $\PP^1$.
Keeping in mind that~\mbox{$\varkappa(S)=1$}, we obtain
\[
-1+\sum (1-1/m_i)=\deg \left(\KKK_{\PP^1}\otimes\mathcal{L}\right)+ \sum (1-1/m_i)> 0.
\]
Hence $\phi$ has at least two multiple fibers. Together with a non-multiple singular fiber
that exists by assumption, this gives a finite $\Gamma$-invariant subset of $\PP^1$
containing at least three points. The latter means that the group $\Gamma$ is finite, which contradicts our assumption.

Therefore, $\phi$ has no non-multiple singular fibers. Hence $\chit(S)=0$, so that Noether's
formula gives~\mbox{$\chi(\OOO_S)=0$}. By the canonical bundle formula one has
\[
-2+\sum (1-1/m_i)> 0,
\]
where $m_i$ are the multiplicities of the multiple fibers of $\phi$.
This implies that $\phi$ has at least three multiple fibers.
Thus we again obtain a finite $\Gamma$-invariant subset in $\PP^1$ containing at
least three points, which gives a contradiction.
\end{proof}

Now we are ready to consider the general case.

\begin{proof}[Proof of Proposition~\ref{proposition:elliptic-fibration-base-subgroup}]
Let $\Gamma$ be the image of the group $\Aut(S)$ in~\mbox{$\Aut(C)$}.
The group $\Gamma$ acts faithfully on the curve $C$.
In particular, if the genus $g(C)$ is at least~$2$, then $\Gamma$ is finite, because the whole
group~\mbox{$\Aut(C)$} is finite.
On the other hand, by Lemma~\ref{lemma:elliptic-fibration-base-subgroup-P1} we can assume that
$g(C)\neq 0$. Therefore, it remains to consider the case~\mbox{$g(C)=1$}.

If $\phi$ is not a smooth morphism, then the group $\Gamma$
has a non-empty finite invariant subset in $C$.
Since~\mbox{$g(C)=1$}, this implies that~$\Gamma$ is finite.

Therefore, we can assume that the morphism $\phi$ is smooth,
and thus all of its fibers are elliptic curves.
In this case we have
$\chit(S)=0$. By Noether's formula one has $\chi(\OOO_S)=0$.
By the canonical bundle formula we obtain
\[
\KKK_S\sim\phi^*\left(\KKK_C\otimes\mathcal{L}\right),
\]
where $\mathcal{L}$ is some line bundle of degree $\chi(\OOO_S)=0$ on $C$.
Hence the Kodaira dimension of $S$ is non-positive, which contradicts our assumption.
\end{proof}

\section{Proof of the main result}
\label{section:proof}

In this section we prove Theorem~\ref{theorem:BFS-for-surfaces}.

\begin{proof}[Proof of Theorem~\ref{theorem:BFS-for-surfaces}]
We may assume that the surface $S$ is minimal.
In this case~\mbox{$\Bim(S)=\Aut(S)$} by Lemma~\ref{lemma:Bir-vs-Aut}.

Suppose that $\varkappa(S)=0$.
Let us use the classification of minimal compact complex surfaces of
Kodaira dimension $0$,
see~\cite[Chapter~VI]{BHPV-2004}: the surface $S$ is either a $K3$ surface, or an Enriques surface,
or a complex torus, or a bielliptic surface, or a Kodaira surface. If $S$ is either a $K3$ surface or an Enriques surface,
then the group $\Aut(S)$ has bounded finite subgroups, see Corollary~\ref{corollary:K3} or~{\cite[Lemma~8.8]{ProkhorovShramov-CCS}}.
On the other hand, if $S$ is either a complex torus, or a bielliptic surface, or
a Kodaira surface, then the group $\Aut(S)$ has unbounded finite subgroups, see~\mbox{\cite[Theorem~1.1(i)]{Shramov-Elliptic}}.

Suppose that $\varkappa(S)=1$.
Consider the pluricanonical map $\phi\colon S\to C$. Since $\phi$ is equivariant with respect to the group $\Aut(S)$,
there is an exact sequence of groups
$$
1\to \Aut(S)_\phi\to \Aut(S)\to\Gamma,
$$
where $\Gamma$ is a subgroup of $\Aut(C)$.
According to Corollary~\ref{proposition:elliptic-fibration-base-subgroup}, the group~$\Gamma$ is finite.
This implies that the group~\mbox{$\Aut(S)$} has bounded finite subgroups if and only if
this holds for the group $\Aut(S)_\phi$.

Finally, if $\varkappa(S)=2$,
then the assertion follows from Theorem~\ref{theorem:general-type}.
\end{proof}

In the case of compact K\"ahler surfaces one can make Theorem~\ref{theorem:BFS-for-surfaces}
a bit more precise. Recall that all Kodaira surfaces are non-K\"ahler.
This follows from the observation that the first Betti number of
a (primary or secondary) Kodaira surface is always odd, see~\cite[\S\,V.5.B.Ib]{BHPV-2004}.
On the other hand, the first Betti number of a compact K\"ahler surface must be even,
see~\cite[Theorem~IV.3.1]{BHPV-2004}.
Thus, Theorem~\ref{theorem:BFS-for-surfaces}
implies

\begin{corollary}\label{corollary:BFS-for-surfaces}
Let $S$ be a compact K\"ahler surface with~\mbox{$\varkappa(S)\ge 0$}.
Suppose that the group $\Bim(S)$
has unbounded finite subgroups. Then~$S$ is bimeromorphic to a surface of one of the following types:
\begin{itemize}
\item a complex torus;

\item a bielliptic surface;

\item a surface of Kodaira dimension $1$.
\end{itemize}
Moreover, in the first two cases the group $\Bim(S)$ always has
unbounded finite subgroups.
\end{corollary}

\begin{remark}
While Theorem~\ref{theorem:BFS-for-surfaces} gives a classification 
of compact complex surfaces of non-negative Kodaira dimension 
whose bimeromorphic automorphism 
groups have bounded finite subgroups, a similar question can be 
asked about automorphism groups of minimal surfaces of negative 
Kodaira dimension. In this case it makes sense to 
investigate finiteness of the whole automorphism group. 
The automorphism groups 
of ruled surfaces were studied in~\cite{Maruyama}; 
in particular, necessary and sufficient conditions 
for their finiteness are known. In higher dimensions, there are 
some results on finiteness of automorphism groups 
of smooth Fano threefolds (see~\cite{CPS})
and certain rational affine varieties 
(see e.g.~\cite{Pukhlikov}).
\end{remark}

\section{Stabilizers of points}
\label{section:stabilizer}

In this section we prove Proposition~\ref{proposition:stabilizer-kappa-0}.
We start by deriving two corollaries from the results of Section~\ref{section:preliminaries}.

\begin{corollary}\label{corollary:bielliptic-stabilizer}
Let $S$ be a bielliptic surface, and let~$P$ be a point on~$S$.
Then the stabilizer $\Aut(S;P)$ has order at most~$36$.
\end{corollary}
\begin{proof}
Consider the Albanese map
$\phi\colon S\to C$.
This map is equivariant with respect to the group $\Aut(S)$.
Recall that $\chit(S)=0$, so that by Noether's formula we have~\mbox{$\chi(\OOO_S)=0$}.
Therefore, it follows from the canonical bundle formula
that $\phi$ has no multiple fibers.

Let $F$ be the fiber of the morphism~$\phi$ passing through~$P$.
Consider the natural homomorphism
$$
\sigma\colon\Aut(S;P)\to\Aut(C)\times\Aut(F).
$$
It follows from Lemma~\ref{lemma:faithful-on-fiber} that the homomorphism $\sigma$ is injective.
Furthermore, its image is contained in the subgroup
$$
\Aut(C;\phi(P))\times\Aut(F;P)\subset\Aut(C)\times\Aut(F).
$$
Since both $C$ and $F$ are elliptic curves, we see that
$$
|\Aut(S;P)|\le |\Aut(C;P)\times\Aut(F;P)|\le 36.
$$
\end{proof}

\begin{corollary}\label{corollary:Kodaira-stabilizer}
Let $S$ be a Kodaira surface, and let~$P$ be a point on~$S$.
Let $G$ be a finite subgroup in the stabilizer $\Aut(S;P)$.
Then $G$ has order at most~$6$ if $S$ is a primary Kodaira surface, and at most $36$ if~$S$ is a secondary Kodaira surface.
\end{corollary}
\begin{proof}
First suppose that $S$ is a primary Kodaira surface. Consider the algebraic reduction~\mbox{$\phi\colon S\to C$},
and set $G'=G\cap\Aut(S)_\phi$.
By Lemma~\ref{lemma:primary-Kodaira-Aut} the subgroup $G'\subset G$
has index at most~$6$.
Let $F$ be a fiber of the morphism $\phi$ passing through the point $P$.
By Lemma~\ref{lemma:primary-Kodaira-fiber} every non-trivial element of the group $G'$ acts on $F$
without fixed points. Therefore, the group $G'$ is trivial, and $|G|\le 6$.

Now suppose that $S$ is a secondary Kodaira surface. Then there exists a canonical finite cover
$$
\theta\colon \tilde{S}\to S,
$$
where $\tilde{S}$ is a primary Kodaira surface, and the degree of~$\theta$ is at most~$6$.
In fact, in this case the  class $[\KKK_S]\in H^2(S,\ZZ)$ of the canonical line bundle $\KKK_S$ is an $n$-torsion element
with $n=2$, $3$, $4$, or $6$, see~\mbox{\cite[\S\,VI.1]{BHPV-2004}}.
By the universal coefficient theorem it defines an $n$-torsion element in~\mbox{$H_1(S,\ZZ)$}
and so there exists a canonically defined subgroup of index~$n$ in the fundamental group $\uppi_1(S)$ which
in turns defines our cover~$\theta$. In other words,
the surface $\tilde S$ is the analytic spectrum $\operatorname{Spec_an}(\mathcal {R})$
of the canonical $\OOO_S$-algebra
$$
\mathcal {R}=\oplus_{i=0}^{n-1} \KKK_S^{\otimes i}.
$$

Since $\theta$ is canonically defined, there is a surjective homomorphism~\mbox{$\Aut(\tilde{S})\to\Aut(S)$}.
In particular, there exists a finite subgroup~\mbox{$\tilde{G}\subset\Aut(\tilde{S})$} with
a surjective homomorphism onto $G$.
Let~$\tilde{P}_0$ be one of the preimages of the point $P$ on $\tilde{S}$. Then~$\tilde{G}$ contains a subgroup~$\tilde{G}_0$ of index at most~$6$
such that $\tilde{G}$ fixes the point~$\tilde{P}_0$. As we already proved above, one has
$|\tilde{G}_0|\le 6$, and hence
$$
|G|\le |\tilde{G}|\le 6|\tilde{G}_0|\le 36.  \qedhere
$$
\end{proof}

\begin{remark}
We do not know if the bounds obtained in Corollaries~\ref{corollary:bielliptic-stabilizer}
and~\ref{corollary:Kodaira-stabilizer} are sharp.
\end{remark}

Now we will complete the case of Kodaira dimension~$0$.

\begin{proof}[Proof of Proposition~\ref{proposition:stabilizer-kappa-0}]
Let $S$ be a compact complex surface of Kodaira dimension~$0$.
Consider the minimal model  $S'$ of the surface~$S$. Then~\mbox{$\Aut(S)\subset\Bim(S')$}; hence by Lemma~\ref{lemma:Bir-vs-Aut} there is
an embedding~\mbox{$\Aut(S)\subset\Aut(S')$}.
Consider the bimeromorphic morphism~\mbox{$\pi\colon S\to S'$}.
Since the minimal model $S'$ is unique, the morphism~$\pi$ is equivariant with respect to the group $\Aut(S)$.
Therefore, the image~\mbox{$\pi(P)$} of the point $P$ is invariant under the group $\Aut(S;P)$.
Thus we can assume from the very beginning that the surface $S$ is minimal.

Let us use the classification of minimal compact complex surfaces of Kodaira dimension $0$.
If $S$ is either a $K3$ surface or an Enriques surface,
the assertion follows from Corollary~\ref{corollary:K3}. If $S$ is a complex torus,
the assertion follows from Corollary~\ref{corollary:torus-stabilizer}.
If $S$ is a bielliptic surface, the assertion follows from Corollary~\ref{corollary:bielliptic-stabilizer}.
Finally, if  $S$ is a Kodaira surface, the assertion follows from Corollary~\ref{corollary:Kodaira-stabilizer}.
\end{proof}

\section{K\"ahler manifolds}
\label{section:Kahler}

In this section we prove Theorem~\ref{theorem:Kahler-stabilizer}.
Given a complex variety $X$, by $\Aut^0(X)$ we will denote the connected component of identity in
the complex Lie group~\mbox{$\Aut(X)$}.

\begin{theorem}[{see \cite[Corollary~5.11]{Fujiki}}]
\label{theorem:Fujiki}
Let $X$ be a compact K\"ahler manifold of non-negative
Kodaira dimension. Then the group
$\Aut^0(X)$ either is trivial, or is a complex torus.
\end{theorem}

\begin{lemma}
\label{lemma:torus-action}
Let $X$ be a compact complex manifold. Suppose that $X$ is non-trivially acted on by a complex torus $T$.
Then $T$ has no fixed points on $X$.
\end{lemma}

\begin{proof}
The action of $T$ on $X$ is given by a morphism
$$
\Psi\colon T\times X\to X.
$$
Suppose that some point $P\in X$
is fixed by $T$.
Then the image~\mbox{$\Psi(T\times\{P\})$} is a point. On the other hand, since the action of
$T$ on $X$ is non-trivial, for a typical point $Q\in X$
the image $\Psi(T\times\{Q\})$ is not a point.
This is impossible by the rigidity theorem, see~\cite[Theorem~5.23]{Kollar-Structure}.
\end{proof}

\begin{corollary}
\label{corollary:torus-action}
Let $X$ be a compact K\"ahler manifold of non-negative Kodaira dimension, and let $P$ be a point on $X$.
Then the group
$$
\Aut^0(X;P)=\Aut(X;P)\cap \Aut^0(X)
$$
is finite. Moreover, there exists a constant $B=B(X)$ that does not depend on $P$,
such that~\mbox{$|\Aut^0(X;P)|\le B$}.
\end{corollary}

\begin{proof}
If the group~\mbox{$\Aut^0(X)$} is trivial, there is nothing to prove.
Thus we will assume that~\mbox{$\Aut^0(X)$} is non-trivial,
and hence it is a complex torus by Theorem~\ref{theorem:Fujiki}.

Suppose that the group $\Aut^0(X;P)$ is infinite.
Since it is a closed subgroup
in~\mbox{$\Aut^0(X)$}, it contains some complex torus~$T$.
Thus~$T$ acts on $X$ with the fixed point $P$,
which contradicts Lemma~\ref{lemma:torus-action}.
Therefore, the group $\Aut^0(X;P)$ is finite.

Now consider the incidence relation
$$
\Xi=\{(\sigma,Q)\mid \sigma(Q)=Q\}\subset \Aut^0(X)\times X,
$$
and denote by $\pi\colon \Xi\to X$ the projection on the second factor.
Then~$\Xi$ is a (possibly reducible) compact complex space, and
a fiber of~$\pi$ over a point $P$ is exactly the subgroup~\mbox{$\Aut^0(X;P)$}.
The projection $\pi$ is a proper  map.
Since the fibers of $\pi$ are finite, we conclude that the map $\pi$ is finite.
We claim that the number of points in
the fibers of $\pi$ is bounded by some constant $B=B(X)$.
Indeed, since the number of irreducible components of $\Xi$ is finite,
we may replace $\Xi$ by its irreducible component. Now the number of points in the fiber
is bounded by the degree of the map $\pi\colon\Xi\to\pi(\Xi)$.
\end{proof}

Now we are ready to prove the main result of this section.

\begin{proof}[Proof of Theorem~\ref{theorem:Kahler-stabilizer}]
According to Corollary~\ref{corollary:torus-action}, the intersection~\mbox{$\Aut(X;P)\cap \Aut^0(X)$} has bounded order.
Hence it is enough to check that the image of $\Aut(X;P)$ in the quotient group
$$
\Upsilon=\Aut(X)/\Aut^0(X)
$$
has bounded finite subgroups.
On the other hand, by \cite[Lemma~3.1]{Kim} the whole group $\Upsilon$ has bounded finite subgroups.
\end{proof}

\begin{corollary}
Let $S$ be a minimal compact K\"ahler surface of Kodaira dimension~$1$.
Suppose that the pluricanonical fibration~\mbox{$\phi\colon S\to C$}
has a singular fiber. Then the group $\Bim(S)$ has bounded finite subgroups.
\end{corollary}

\begin{proof}
One has~\mbox{$\Bim(S)=\Aut(S)$} by Lemma~\ref{lemma:Bir-vs-Aut}.
Consider the set~\mbox{$\Sigma\subset S$} of singular points of all singular fibers of $\phi$, and choose a point~\mbox{$P\in\Sigma$}.
Since $\phi$ is equivariant with respect to the group
$\Aut(S)$, this group acts on $\Sigma$.
Hence the stabilizer $\Aut(S;P)$ of the point $P$ has index at most $|\Sigma|$ in $\Aut(S)$.
Now it remains to apply Theorem~\ref{theorem:Kahler-stabilizer}.
\end{proof}

\section{Abelian subgroups}
\label{section:Jordan}

In this section we prove Proposition~\ref{proposition:uniformly-Jordan}.
We start with one of its particular cases that we treat using the results of Section~\ref{section:preliminaries}.

\begin{corollary}\label{corollary:Kodaira-abelian}
Let $S$ be a Kodaira surface, and let $G$ be a finite subgroup in $\Aut(S)$.
Then $G$ contains an abelian subgroup of index at most $6$.
\end{corollary}
\begin{proof}
First suppose that $S$ is a primary Kodaira surface. Consider the algebraic reduction~\mbox{$\phi\colon S\to C$},
and set $G'=G\cap\Aut(S)_\phi$.
By Lemma~\ref{lemma:primary-Kodaira-Aut} the subgroup $G'\subset G$
has index at most~$6$.
By Corollary~\ref{corollary:Kodaira-stabilizer} the group $G'$ acts by translations on an elliptic curve
that is a fiber of the morphism $\phi$. In particular, this group is abelian.

Now suppose that $S$ is a secondary Kodaira surface. As in the proof of Corollary~\ref{corollary:Kodaira-stabilizer},
there exists a primary Kodaira surface $\tilde{S}$ and a finite subgroup~\mbox{$\tilde{G}\subset\Aut(\tilde{S})$}
with a surjective homomorphism onto $G$.
According to what we proved above, the group $\tilde{G}$ contains an abelian subgroup $\tilde{G}'$ of index at most $6$.
Its image in $G$ will also be an abelian subgroup of index at most~$6$.
\end{proof}

Now we prove a general assertion.

\begin{proof}[Prove of Proposition~\ref{proposition:uniformly-Jordan}]
Let $S$ be a compact complex surface of Kodaira dimension~$0$.
Replace $S$ by its minimal model. Then~\mbox{$\Bim(S)=\Aut(S)$} by Lemma~\ref{lemma:Bir-vs-Aut}.

Let us use the classification of minimal compact complex surfaces
of Kodaira dimension $0$.
If $S$ is either a $K3$ surface or an Enriques surface, then the assertion follows from Corollary~\ref{corollary:K3}.
If $S$ is a complex torus, then the assertion follows from Corollary~\ref{corollary:torus-stabilizer},
since in this case the quotient of the group $\Aut(S)$ by the abelian subgroup acting by translations on $S$
is isomorphic to the stabilizer of (any) point of $S$.
If~$S$ is a Kodaira surface, then the assertion follows from Corollary~\ref{corollary:Kodaira-abelian}.
Finally, if $S$ is a bielliptic surface, then the assertion follows from the explicit
description of the group $\Aut(S)$, see~\mbox{\cite[Table~3.2]{BennettMiranda}}: according to this description
$\Aut(S)$ contains a subgroup of index at most~$24$ isomorphic to the group of points of an elliptic curve.
\end{proof}


\begin{thebibliography}{10}
\providecommand*{\href}[2]{{\small #2}}
\providecommand*{\url}[1]{{\small #1}}
\providecommand*{\BibUrl}[1]{\url{#1}}
\providecommand{\BibAnnote}[1]{}
\providecommand*{\BibEmph}[1]{#1}
\ProvideTextCommandDefault{\cyrdash}{\iflanguage{russian}{\hbox
  to.8em{--\hss--}}{\textemdash}}
\providecommand*{\BibDash}{\ifdim\lastskip>0pt\unskip\nobreak\hskip.2em plus
  0.1em\fi
\cyrdash\hskip.2em plus 0.1em\ignorespaces}
\renewcommand{\newblock}{\ignorespaces}


\bibitem{Akhiezer}
\BibEmph{Akhiezer~D.~N.} \href{http://dx.doi.org/10.1007/978-3-322-80267-5}{Lie
  group actions in complex analysis}. {Aspects of Mathematics, E27}. \BibDash
\newblock Braunschweig: Friedr. Vieweg \& Sohn, 1995. \BibDash
\newblock P.~viii+201. \BibDash
\newblock
  ISBN:~\href{http://isbndb.com/search-all.html?kw=3-528-06420-X}{3-528-06420-X}.
  \BibDash
\newblock Access mode: \BibUrl{http://dx.doi.org/10.1007/978-3-322-80267-5}.

\bibitem{BennettMiranda}
\BibEmph{Bennett~C., Miranda~R.} The automorphism groups of the hyperelliptic
  surfaces~// \href{http://dx.doi.org/10.1216/rmjm/1181073156}{\BibEmph{Rocky
  Mountain J. Math.}} \BibDash
\newblock 1990. \BibDash
\newblock Vol.~20, no.~1. \BibDash
\newblock P.~31--37. \BibDash
\newblock Access mode: \BibUrl{https://doi.org/10.1216/rmjm/1181073156}.

\bibitem{Birkar}
\BibEmph{Birkar~C.} Singularities of linear systems and boundedness of {F}ano
  varieties~// \BibEmph{arXiv:1609.05543}. \BibDash
\newblock 2016.

\bibitem{BHPV-2004}
Compact complex surfaces~/ W.~P.~Barth, K.~Hulek, Ch. A.~M.~Peters,
  A.~Van~de~Ven. \BibDash
\newblock Second edition. \BibDash
\newblock Berlin~: Springer-Verlag, 2004. \BibDash
\newblock Vol.~4 of \BibEmph{Ergebnisse der Mathematik und ihrer Grenzgebiete.
  3. Folge. A Series of Modern Surveys in Mathematics}. \BibDash
\newblock P.~xii+436. \BibDash
\newblock
  ISBN:~\href{http://isbndb.com/search-all.html?kw=3-540-00832-2}{3-540-00832-2}.

\bibitem{CPS}
\BibEmph{Cheltsov~I., Przyjalkowski~V., Shramov~C.} 
Fano threefolds with infinite automorphism groups~//  
\BibEmph{Izv. Math.} \BibDash
\newblock 2019. \BibDash 
\newblock Vol.~83, no.~4. \BibDash
\newblock P.~860--907. 



\bibitem{Fujiki}
\BibEmph{Fujiki~A.} On automorphism groups of compact {K}{\"a}hler manifolds~//
  \BibEmph{Invent. Math.} \BibDash
\newblock 1978. \BibDash
\newblock Vol.~44, no.~3. \BibDash
\newblock P.~225--258.

\bibitem{Huybrechts}
\BibEmph{Huybrechts~D.}
  \href{http://dx.doi.org/10.1017/CBO9781316594193}{Lectures on {K}3 surfaces}.
  \BibDash
\newblock Cambridge University Press, Cambridge, 2016. \BibDash
\newblock Vol.~158 of \BibEmph{Cambridge Studies in Advanced Mathematics}.
  \BibDash
\newblock P.~xi+485. \BibDash
\newblock
  ISBN:~\href{http://isbndb.com/search-all.html?kw=978-1-107-15304-2}{978-1-107-15304-2}.
  \BibDash
\newblock Access mode: \BibUrl{https://doi.org/10.1017/CBO9781316594193}.

\bibitem{Kim}
\BibEmph{Kim~J.~H.} Jordan property and automorphism groups of normal compact
  {K}\"ahler varieties~//
  \href{http://dx.doi.org/10.1142/S0219199717500249}{\BibEmph{Commun. Contemp.
  Math.}} \BibDash
\newblock 2018. \BibDash
\newblock Vol.~20, no.~3. \BibDash
\newblock P.~1750024, 9.

\bibitem{Kodaira-1963}
\BibEmph{Kodaira~K.} On compact analytic surfaces. {II}~// \BibEmph{Ann. of
  Math. (2)}. \BibDash
\newblock 1963. \BibDash
\newblock Vol.~77. \BibDash
\newblock P.~563--626.

\bibitem{Kollar-Structure}
\BibEmph{Koll\'{a}r~J.} The structure of algebraic threefolds: an introduction
  to {M}ori's program~//
  \href{http://dx.doi.org/10.1090/S0273-0979-1987-15548-0}{\BibEmph{Bull. Amer.
  Math. Soc. (N.S.)}}. \BibDash
\newblock 1987. \BibDash
\newblock Vol.~17, no.~2. \BibDash
\newblock P.~211--273. \BibDash
\newblock Access mode:
  \BibUrl{https://doi.org/10.1090/S0273-0979-1987-15548-0}.

\bibitem{Maruyama}
\BibEmph{Maruyama~M.}
On automorphism groups of ruled surfaces~//
\BibEmph{J. Math. Kyoto Univ.} \BibDash
\newblock 1971. \BibDash 
\newblock Vol.~11. \BibDash
\newblock 89--112. 


\bibitem{Miranda1989}
\BibEmph{Miranda~R.} The basic theory of elliptic surfaces. {Dottorato di
  Ricerca in Matematica. [Doctorate in Mathematical Research]}. \BibDash
\newblock ETS Editrice, Pisa, 1989. \BibDash
\newblock P.~vi+108.

\bibitem{MukaiNamikawa}
\BibEmph{Mukai~Sh., Namikawa~Yu.} Automorphisms of {E}nriques surfaces which
  act trivially on the cohomology groups~//
  \href{http://dx.doi.org/10.1007/BF01388829}{\BibEmph{Invent. Math.}} \BibDash
\newblock 1984. \BibDash
\newblock Vol.~77, no.~3. \BibDash
\newblock P.~383--397. \BibDash
\newblock Access mode: \BibUrl{https://doi.org/10.1007/BF01388829}.

\bibitem{Nikulin}
\BibEmph{Nikulin~V.} Classification of Picard lattices of $K3$-surfaces~//
\BibEmph{Izv. Math.} \BibDash
\newblock 2018. \BibDash
\newblock Vol.~82, no.~4. \BibDash
\newblock P.~752--816.


\bibitem{ProkhorovShramov-Bir}
\BibEmph{Prokhorov~Yu., Shramov~C.} Jordan property for groups of birational
  selfmaps~//
  \href{http://dx.doi.org/10.1112/S0010437X14007581}{\BibEmph{Compositio
  Math.}} \BibDash
\newblock 2014. \BibDash
\newblock Vol. 150, no.~12. \BibDash
\newblock P.~2054--2072.

\bibitem{ProkhorovShramov-CCS}
\BibEmph{Prokhorov~Yu., Shramov~C.} Automorphism groups of compact complex
  surfaces~// 
\url{https://doi.org/10.1093/imrn/rnz124}   
\BibDash
\newblock 2020.

\bibitem{Prokhorov-Shramov-3folds}
\BibEmph{Prokhorov~Yu., Shramov~C.} Finite groups of birational selfmaps of
  threefolds~//
  \href{http://dx.doi.org/10.4310/MRL.2018.v25.n3.a11}{\BibEmph{Math. Res.
  Lett.}} \BibDash
\newblock 2018. \BibDash
\newblock Vol.~25, no.~3. \BibDash
\newblock P.~957--972.

\bibitem{Pukhlikov}
\BibEmph{Pukhlikov~A.}
Automorphisms of certain affine complements in the projective space~//
\BibEmph{Sb. Math.} \BibDash
\newblock 2018. \BibDash
\newblock Vol.~209, no.~2. \BibDash
\newblock P.~276--289. 



\bibitem{Sawon}
\BibEmph{Sawon~J.} Isotrivial elliptic {$K3$} surfaces and {L}agrangian
  fibrations~//
\BibEmph{arXiv:1406.1233}. \BibDash
\newblock 2014.

\bibitem{Serre2007}
\BibEmph{Serre~J.-P.} Bounds for the orders of the finite subgroups of
  {$G(k)$}~// Group representation theory. \BibDash
\newblock EPFL Press, Lausanne, 2007. \BibDash
\newblock P.~405--450.

\bibitem{Shramov-Elliptic}
\BibEmph{Shramov~C.} Finite groups acting on elliptic surfaces~//
\url{https://doi.org/10.1007/s40879-019-00383-y} 
\BibDash
\newblock 2019.

\bibitem{ShramovVologodsky}
\BibEmph{Shramov~C., Vologodsky~V.} Automorphisms of pointless surfaces~//
\BibEmph{arXiv:1807.06477}. \BibDash
\newblock 2018.


\bibitem{Ueno1975}
\BibEmph{Ueno~K.} Classification theory of algebraic varieties and compact
  complex spaces. Lecture Notes in Mathematics, Vol. 439. \BibDash
\newblock Springer-Verlag, Berlin-New York, 1975. \BibDash
\newblock P.~xix+278. \BibDash
\newblock Notes written in collaboration with P.~Cherenack.

\end{thebibliography}
\bibliographystyle{ugost2008s}

\end{document}